\documentclass[	
	paper = a4,
	fontsize = 11pt,
	final, 
	DIV = 14,
	headings = normal,
	abstract = true
]{scrartcl}

\usepackage[english]{babel}

\usepackage{cmap}       
\usepackage[T1]{fontenc}
\usepackage[utf8]{inputenc}

\usepackage{exscale}
\usepackage[l2tabu, orthodox]{nag}
\usepackage{etex}
\reserveinserts{28}	
\usepackage{fixltx2e}

\usepackage[pdftex]{graphicx}
\usepackage[citestyle=authoryear, natbib=true, maxbibnames=99, mincitenames=3, maxcitenames=3, useprefix=true, backend=bibtex]{biblatex}
\usepackage{subfigure}

\usepackage{mdframed}
\usepackage[auth-lg]{authblk}
\usepackage{amsmath}
\usepackage{amssymb}
\usepackage{amsthm}
\usepackage{amsfonts}
\usepackage{thmtools}			
\usepackage{thm-restate}
\usepackage[mathic=true]{mathtools}
\usepackage[all, warning]{onlyamsmath}
\usepackage{fixmath}

\usepackage{enumitem}
\setlist[enumerate]{noitemsep, topsep=0.5\topsep}
\setlist[description]{noitemsep, topsep=0.5\topsep}
\setlist[itemize]{noitemsep, topsep=0.5\topsep}

\theoremstyle{plain}
\newtheorem{theorem}{Theorem}[section]
\newtheorem{proposition}[theorem]{Proposition}
\newtheorem{lemma}[theorem]{Lemma}
\newtheorem{corollary}[theorem]{Corollary}
\theoremstyle{definition}
\newtheorem{definition}[theorem]{Definition}

\usepackage{ellipsis}
\usepackage{xspace}
\usepackage{hfoldsty}
\usepackage{lmodern}

\usepackage{algorithm}
\usepackage{algpseudocode}
\algnotext{EndFor}
\algnotext{EndIf}
\algnotext{EndProcedure}
\algnotext{EndFunction}
\algnotext{EndWhile}

\addtokomafont{disposition}{\rmfamily}
\addtokomafont{descriptionlabel}{\rmfamily}

\newcommand{\cA}{\ensuremath{{\mathcal A}}\xspace}

\newcommand{\cO}{\ensuremath{{\mathcal O}}\xspace}

\newcommand{\cQ}{\ensuremath{{\mathcal Q}}\xspace}

\newcommand{\cY}{\ensuremath{{\mathcal Y}}\xspace}

\newcommand{\bbR}{\ensuremath{\mathbb{R}}}
\newcommand{\bbQ}{\ensuremath{\mathbb{Q}}}
\newcommand{\bbZ}{\ensuremath{\mathbb{Z}}}
\newcommand{\bbN}{\ensuremath{\mathbb{N}}}

\newcommand{\cp}{\textbf{P}\xspace}
\newcommand{\cnp}{\textbf{NP}\xspace}
\newcommand{\cconp}{\textbf{co-NP}\xspace}

\newcommand{\totalp}{\textbf{TotalP}\xspace}
\newcommand{\incp}{\textbf{IncP}\xspace}
\newcommand{\delayp}{\textbf{DelayP}\xspace}
\newcommand{\delayps}{\textbf{PSDelayP}\xspace}

\newcommand{\RR}{\mathbb{R}}
\newcommand{\QQ}{\mathbb{Q}}

\newcommand{\NN}{\mathbb{N}}

\newcommand{\rb}{\right\}\xspace}
\newcommand{\lb}{\left\{\xspace}
\newcommand{\lbr}{\left(\xspace}
\newcommand{\rbr}{\right)\xspace}

\newcommand{\cosolutions}{\mathcal{S}\xspace}
\newcommand{\encsize}[1]{\langle #1 \rangle\xspace}

\newcommand{\finished}[1]{{#1}^{\textsc{Fin}}}
\newcommand{\decision}[1]{{#1}^{\textsc{Dec}}}

\DeclareMathOperator{\conv}{conv}

\DeclareMathOperator{\poly}{poly}

\DeclareMathOperator{\lex}{lex}
\DeclareMathOperator{\lexmin}{lexmin}

\newcommand{\YN}{\cY_{\text{N}}\xspace}
\newcommand{\YX}{\cY_{\text{X}}\xspace}

\newcommand{\Y}{\cY\xspace}

\newcommand{\ws}[2]{{#1}^{\textsc{WS}}(#2)}
\newcommand{\lws}[2]{{#1}^{\textsc{Lex-WS}}(#2)}

\newcommand{\enc}[1]{\ensuremath{\langle #1 \rangle}}
\newcommand{\trans}{^T}

\usepackage{color}
\newcommand{\new}[1]{\emph{#1}}

\begin{document}

\title{\normalfont \LARGE Output-sensitive Complexity of Multiobjective Combinatorial Optimization}

\author[1]{Fritz Bökler\footnote{Correspondence to: Department of Computer Science, TU Dortmund, Otto-Hahn-Str. 14, 44225 Dortmund, Germany. E-mail: fritz.boekler@tu-dortmund.de. The author has been supported by the Bundesministerium für Wirtschaft und Energie (BMWi) within the research project ``Bewertung und Planung von Stromnetzen'' (promotional reference 03ET7505) and by DFG GRK 1855 (DOTS).}}
\author[2]{Matthias Ehrgott}
\author[1]{Christopher Morris\footnote{The author is funded by the German Science Foundation (DFG) within the Collaborative Research Center SFB 876 “Providing Information by Resource-Constrained Data Analysis”, project A6 “Resource-efficient Graph Mining”.}}
\author[1]{Petra Mutzel}
\affil[1]{Department of Computer Science\\TU Dortmund University, Dortmund, Germany\\

 \texttt{\{fritz.boekler, christopher.morris, petra.mutzel\}@tu-dortmund.de}}
\affil[2]{Department of Management Science, Lancester University, Lancester, United Kingdom\\

 \texttt{m.ehrgott@lancaster.ac.uk}}

\date{\vspace{-40pt}}

\maketitle

\begin{abstract}
We study output-sensitive algorithms and complexity for multiobjective combinatorial optimization problems.
In this computational complexity framework, an algorithm for a general enumeration problem is regarded efficient if it is output-sensitive, i.e., its running time is bounded by a polynomial in the input and the output size.
We provide both practical examples of MOCO problems for which such an efficient algorithm exists as well as problems for which no efficient algorithm exists under mild complexity theoretic assumptions.

\end{abstract}

\section{Introduction}

Computational complexity theory in multiobjective optimization has been considered in different shapes for quite a while.
In this paper, we argue that in contrast to the traditional way of investigating the complexity of multiobjective optimization problems, especially multiobjective combinatorial optimization (MOCO) problems, output-sensitive complexity theory yields deeper insights into the complexity of these problems.

We stress here, that our discussions and results apply not only to MOCO problems, but also to every multiobjective optimization problem with a well-defined and finite set to output.
For example, they also apply to multiobjective integer optimization problems with finite Pareto-fronts.

The problem on which we will exemplify our considerations is the following.
\begin{definition}[Multiobjective Combinatorial Optimization Problem] \label{def:moco}
	An \new{instance of a multiobjective combinatorial optimization problem} is a pair $(\cosolutions, C)$, where
	\begin{itemize}
		\item $\cosolutions \subseteq \{0,1\}^n$ is the set of \emph{solutions}, and
		\item $C \in \QQ^{d \times n}$ is the \emph{objective function matrix}.
	\end{itemize}
	The \new{multiobjective combinatorial optimization problem} consists of all its instances.
	The goal is to enumerate for a given instance $(\cosolutions, C)$ all minimal elements of $\Y = \{ Cx \mid x \in \cosolutions\}$ with respect to the canonical partial order on vectors.
\end{definition}
The canonical partial order on vectors is the componentwise less-or-equal partial order, denoted as $\leq$.
We call the minima of $\Y$ the \emph{Pareto-front}, which we denote by $\YN$.
Points in $\YN$ are called \new{nondominated} points.
The preimage of the Pareto-front is called \emph{Pareto-set}; solutions in the Pareto-set are called \new{Pareto-optimal} or \new{efficient} solutions.
Note that the MOCO problem defined as above is different from the problem of enumerating the Pareto-set.

In the definition of the MOCO problem, we only allow rational numbers in the input as real numbers cannot be encoded in our model of computation.
But of course, if we show that these problems are already hard on rational numbers, they do not become easier when extending the set of numbers allowed.

When investigating the complexity of these problems, it is important to define the input size precisely.
As $\cosolutions$ can be very large, it is supposed to be implicitly given.
This accounts for the fact that searching $\cosolutions$ exhaustively is explicitly not what we want to do.
Thus, the input size is supposed to be $n$ plus the encoding length of the matrix $C$, denoted as $\langle C \rangle$, which will be defined in the preliminaries in Section~\ref{sec:preliminaries}.

In practice we also want to find a solution $x\in\cosolutions$ which maps to a point $y\in \YN$.
Regarding the formal definition above, this renders the problem to not be well-defined.
Thus, when describing an algorithm, we usually expect that also a solution is produced for each point of the Pareto-front.
When we prove a hardness result, Definition~\ref{def:moco} gives us problems which are not harder than a variant of the problem where we also want to find solutions.
Hence, proving hardness of the problem as defined in Definition~\ref{def:moco} will lead to a hardness result for the problem including finding of a representative solution.

\subsection{Traditional Complexity Results for MOCO}
In the past, complexity of MOCO problems has been investigated in the sense of $\cnp$-hardness.
The canonical decision problem of a MOCO problem is as follows.
\begin{definition}[Canonical Decision Problem] \label{def:can_dec_problem}
Given a MOCO problem $P$, the \emph{canonical decision problem} $\decision{P}$ is defined as follows:
We are given an instance $(\cosolutions, C)$ of $P$ and a vector $k\in \QQ^d$ and the goal is to decide whether there exists an $x\in\cosolutions$, such that $Cx\leq k$.
The input size is supposed to be $n + \encsize{C} + \encsize{k}$.
\end{definition}
Obviously, any MOCO problem that is $\cnp$-hard for $d=1$ is also $\cnp$-hard for $d\geq 2$.
Moreover, known results in the literature show that even for MOCO problems with $d=2$ objectives $\decision{P}$ is $\cnp$-hard, even for problems that are in $\cp$ when $d=1$ such as the biobjective shortest path problem~\citep{serafini}, biobjective minimum spanning tree problem~\citep{cametal}, biobjective assignment problem~\citep{serafini}, and biobjective uniform matroid problem~\citep{preholzhau}.
The proof of Proposition~\ref{thm:trad_hard_enum_easy} also shows that $\decision{P}$ for a MOCO problem with two objectives and no constraints is $\cnp$-complete.

One remark is in order here.
It is not clear if the $\cnp$-hardness of $\decision{P}$ implies the $\cnp$-hardness of the MOCO problem $P$.
In fact, \citet{serafini} proved the polynomial-time equivalence of $P$ and $\decision{P}$ only for the case where $||Cx - Cx'||_\infty < \delta n^l$ for every two solutions $x, x' \in \cosolutions$ and some $\delta,l \geq 1$.

Furthermore, we note that $\decision{P}$ for $d=2$ is the same as the decision problem of the so-called \new{resource-constrained combinatorial optimization problem}
\begin{align*}
(\cosolutions, c, r, \hat r),
\end{align*}
where $\cosolutions$ is as in Definition~\ref{def:moco}, $c \in \QQ^n$, $r \in \QQ^n$, $\hat r \in \QQ$, and the goal is to find a minimizer of $\{c^T x \mid x \in \cosolutions,\, r^Tx \le  \hat r\}$.
This observation reveals that the canonical decision problem for a MOCO problem with $d=2$ and that for a related resource-constrained combinatorial optimization problem are identical, despite the fundamental difference in the goals of these two problems.
We can therefore question, whether studying the $\cnp$-hardness of $\decision{P}$ can provide much insight in the hardness of MOCO problems on top of insights into the hardness of resource-constrained single objective combinatorial optimization problems.

Another important consideration is the size of the Pareto-front.
Many researchers have provided instances of MOCO problems for which the size of $\YN$ is exponential, we refer to~\citep{H79} for the biobjective shortest path problem,~\citep{hamruh} for the biobjective minimum spanning tree problem,~\citep{ruhe88} for the biobjective integer minimum-cost flow problem, and~\citep{Ehr2005} for the biobjective unconstrained combinatorial optimization problem.
This consideration also raises the question, if the $\cnp$-hardness of a MOCO problem $P$ which has a Pareto-front of exponential size, does add information to the picture.
Because $\cnp$-hardness is still only an indication that $P$ cannot be solved in polynomial time, but $P$ having a Pareto-front of exponential size is an actual proof of this fact. Despite these negative results, a few results on MOCO problems with a Pareto-front of polynomial size or which belong to $\cp$ are known.
We refer to~\citep{FFHKPRSSW16} in this issue for more details.

This brief overview on complexity results for MOCO problems reveals what one could consider a bleak picture:
Multiobjective versions of even the simplest combinatorial optimization problems have Pareto-fronts of exponential size, and thus there is no prospect for finding polynomial time algorithms to determine the Pareto-front.
A closer look at some of the pathological biobjective instances does, however, reveal, that despite the exponential size of the Pareto-front, all elements of the Pareto-front lie on a single line in $\QQ^2$.
Hence, we might hypothesize that traditional worst case complexity analysis, which measures complexity in relation to input size, is not the right tool to investigate the complexity of MOCO problems.
Instead, we propose to investigate output-sensitive complexity.

\subsection{The Pareto-front Size within the Smoothed Analysis Framework}
Before we move forward to output-sensitive complexity, we need to adress one practical objection regarding MOCO problems.
It is often argued that computing the entire Pareto-front of a MOCO problem is too costly to pursue, because it can have exponential size in the worst-case.
But in practice, it was observed that usually the situation is not so bad when the number of objectives is small.
This observed discrepancy between practice and the traditional worst case analysis motivates stochastic running time analysis, where the inputs are drawn from a certain distribution.
One prominent stochastic running time analysis framework, so called \emph{Smoothed Analysis}, can be used to bound the expected worst-case size of the Pareto-front of a MOCO problem.

In classical running time analysis, we play against an adversary who gives us ill posed instances in the sense that the running time is very high or the Pareto-front is very large.
Smoothed Analysis aims at weakening this adversary.
In the context of multiobjective optimization, a model by~\citet{ANRV05} was used to show the best expected bound on the smoothed worst-case size of the Pareto-front of general MOCO problems.

Instead of all the objective function matrix entries, the adversary is only allowed to choose the first row of it deterministically and for all remaining rows $i\in\{2,\dots, d\}$ and columns $j\in [n]$ it may choose a probability density function $f_{i,j}: [0,1] \rightarrow \RR$, describing how potential entries are drawn.
The adversary is not allowed to choose these densities arbitrarily, because otherwise it could again choose the coefficients deterministically.
Rather, the $f_{i,j}$ are bounded by a model parameter $\phi$.

Consequently, the adversary gives us a set of solutions $\cosolutions \subseteq \{0,1\}^n$, the first row of the objective function matrix $c\in\QQ^n$ and the probability densities $f_{i,j}$.
\citet{BR12} showed that the expected size of the Pareto-front of these instances is at most $\mathcal{O}(n^{2d}\phi^d)$.

Thus, when we fix the number of objectives, we can expect to have only polynomial Pareto-front sizes of any MOCO problem in practice.
This emphasizes the need for output-sensitive algorithms, because when the Pareto-front is small, we want our algorithms to be fast and when the Pareto-front is large, we want them to be not too slow.

\subsection{Organization}

In the remainder of the paper, we will first give some definitions in Section~\ref{sec:preliminaries}.
In particular, we will give necessary definitions from theoretical computer science, including output-sensitive complexity.

Afterwards, we will be concerned with a sufficient condition for a MOCO problem to be efficiently solvable in the sense of output-sensitive complexity in Section~\ref{sec:suff_cond}.
This sufficient condition can be applied to the multiobjective global minimum-cut problem which will be our first example of an efficiently solvable MOCO problem.

We will show that the multiobjective shortest path problem is an example of a problem which is not efficiently solvable under weak complexity theoretic assumptions in Section~\ref{sec:mosp}.
Moreover, we will demonstrate a general method from output-sensitive complexity for showing that an enumeration problem is not efficiently solvable under the assumption that $\cp \neq \cnp$.

In Section~\ref{sec:supported}, we will again be concerned with efficiently solvable multiobjective optimization problems.
We will discuss some connections between MOCO problems and multiobjective linear programming.
We will show that the impression that computing extreme points of the Pareto-front is easier than computing the whole Pareto-front is indeed true in the world of output-sensitive complexity, at least for the multiobjective shortest path problem.
Moreover, we review some recent results from the literature about output-sensitive algorithms for computing supported solutions and extreme nondominated points and also provide some new results for the special case of $d=2$.

\section{Preliminaries} \label{sec:preliminaries}

In the following section we fix notation. Let $v$ be a vector in $\bbR^n$, then we denote its $i$-th \new{component} by $v_i$. We assume that all vectors are \new{column vectors}. The $i$-th unit vector is denoted by $e_i$. Let $M$ be a matrix in $\bbR^{m \times n}$, then we denote its $i$-th \new{row} by $M_i$ and the \new{scalar} in the $i$-th row and the $j$-th column by $m_{ij}$. The \new{transpose}\index{Transpose} of $M$ is denoted by $M\trans$\index{$M\trans$}. For $r$ in $\bbR$, the \new{ceiling function} is $\lceil r \rceil := \min \lb z \in \bbZ \mid z \geq r \rb$ and let $[n] := \lb 0, \dots, n \rb \subset \bbN$ and $[n\!:\! m] := [n,m] \cap \bbN$.

Let $(S, \preceq)$ be a partially ordered set.
Then $s <_{\lex} t$ for $s$ and $t$ in $S^p$, $p\in\NN$, if there is $j$ in $[1\!:\!p]$ such that
\begin{enumerate}
	\item $s_i = t_i$ for $i < j$, and
	\item $s_j < t_j\,.$
\end{enumerate}

\subsection{Theoretical Computer Science}
In theoretical computer science, it is important to define the model of computation used for running time analyses.
In the complexity of enumeration problems, the model of computation usually is the \emph{Random Access Machine} (RAM).
It can be concieved as a formal model of a computer, with a fixed and small set of instructions and an infinite number of memory cells.
These memory cells are only allowed to hold integer numbers with up to $\mathcal{O}(n^l)$ bits on inputs of size $L$ for some constant $l \geq 1$ to account for the precision needed in optimization.
For more details, we refer the reader to the book by \citet[23~sqq.]{Cor+2001}.

Moreover, we also need to discuss encoding lengths of numbers.
The following is largely based on~\citet[29~sqq.]{Gro+1988}. For $z \in \bbZ \setminus \{ 0 \}$, the \new{encoding length} of $z$, i.e., an upper bound for the minimum number of bits to encode $z$ in binary representation, is 
\begin{equation*}\label{e:en}
\enc{z} := 1+ \lceil \log_2 (|z| + 1) \rceil\,.\index{$\enc{z}$}
\end{equation*}
That is, one bit to represent the sign and $\lceil \log_2 (|z| + 1) \rceil$ bits to encode the binary representation of the absolute value of $z$.
For $z=0$ one bit is sufficient, i.e., $\enc{0} := 1$.
Let $q \in \bbQ$ and let $q = a/b$ for $a$ and $b \in \bbZ$ such that $a$ and $b$ are \emph{coprime}, and $b > 0$, then the encoding length of $q$ is
\begin{equation*}\label{e:eq}
\enc{q} := \enc{a} + \enc{b}\,.
\end{equation*}
Let $v$ be a vector in $\bbQ^n$, then $\enc{v} := n +\sum_{i=1}^n \enc{v_i}$.
Moreover, let $M$ be a matrix in $\bbQ^{m \times n}$, then $\enc{M} := mn + \sum_{i=1}^m \enc{M_i}$.

Due to the special importance of polynomial running time, we define the set of functions not growing faster than a polynomial function as follows:
Let $f\colon \bbR^d \to \bbR$ be a function, then $f$ is in $\poly(n^1, \dots, n^d)$\index{$\poly(n^1, \dots, n^d)$} if there is a polynomial function $p\colon \bbR^d \to \bbR$ such that $f$ in $\cO(p(n^1, \dots, n^d))$. For a more concise introduction to formal languages and complexity theory, e.g., \cnp and \cconp, we refer to~\citep{Aro+2009}.


\subsection{Introduction to Output-sensitive Complexity}

Classically, the running time of an algorithm is measured with respect to the input size. An algorithm which has a running time which is polynomially-bounded in the input size is widely regarded as efficient. For certain problems, such as MOCO problems, the size of the output varies widely or even becomes exponentially large in the input size. 

Here, the notion of output-sensitive algorithms comes into play, where running time is not only measured in the input size but also in the size of the output, cf.~\citep{Joh+1988} for examples of output-sensitive algorithms. The point here is that we can view the MOCO problem of Definition~\ref{def:moco} as an \emph{enumeration problem}, i.e., enumerating the elements of the (finite) Pareto-front, see Section~\ref{mocoenum}. This task is solved by an \emph{enumeration algorithm}, i.e., an algorithm that outputs every element of the the Pareto-front exactly once.

In the following, we give a \emph{formal} definition of the terms enumeration problem and enumeration algorithm.
Moreover, we define complexity classes for enumeration problems. This section is largely based on~\citep{Schm2009}.

\begin{definition}[{cf.~\citep[7~sq.]{Schm2009}}]\label{d:ep}
	An \new{enumeration problem}\index{Enumeration Problem} is a pair $\lbr I, C \rbr$, such that 
	\begin{enumerate}
		\item $I \subseteq \Sigma^*$ is a language for some fixed alphabet $\Sigma$,  
		\item $C \colon I \to \Sigma^*$ maps each \new{instance}\index{Instance}  $x$ in $I$ to its \new{configurations} $C(x)$, and
		\item the encoding length $\enc{s}$ for $s$ in $C(x)$ for $x$ in $I$ is in $\poly(x)$.
	\end{enumerate}
	We assume that $I$ is decidable in polynomial time and that $C$ is computable.
\end{definition}

If the reader is not familiar with the term language, the set $\Sigma^*$ can be interpreted as the set of all finite strings over $\{0,1\}$.
This is important, because all instances and all configurations need to be encoded by finite strings, while the exact encoding is not relevant here.
The third requirement means, that the configurations need to be compact, i.e., the encoding length of a configuration of an instance $x$ should be polynomially bounded in the size of $x$.
The reason for this is to limit the number of configurations to be at most exponentially many.

\begin{definition}[cf.~{\citep[8]{Schm2009}}]\label{d:ea} 
	Let $E = (I, C)$ be an enumeration problem. An \new{enumeration algorithm} for $E$ is a RAM that
	\begin{enumerate}
		\item on input $x$ in $I$ outputs $c$ in $C(x)$ \emph{exactly once}, and
		\item on every input terminates after a \emph{finite} number of steps.
	\end{enumerate}
	Let $x$ in $I$ and let $|C(x)|= k$, then the $0$-th \new{delay} is the time before the first solution is output, the $i$-th delay\index{Delay} for $i$ in $[1\!:\!k-1]$ is the time between the output of the $i$-th and the $(i+1)$-th solution, and the $k$-th delay is the time between the last output and the termination of the algorithm.\end{definition}

The following complexity classes can be used to classify enumeration problems, cf.~\citep{Joh+1988}.

\begin{definition}[{cf.~\citep[12]{Schm2009}}]\label{d:pdt}  
	Let $E = (I, C)$ be an enumeration problem. Then $E$ is in 
	\begin{enumerate}
		\item $\totalp$\index{$\totalp$} (Output-Polynomial Time/Polynomial Total Time)\footnote{Although we use the term \emph{output-polynomial time} in the remaining text, we abbreviate it to $\totalp$ for historical and notational reasons.},
		\item $\incp$\index{$\incp$}  (Incremental Polynomial Time)\index{Incremental Polynomial Time},
		\item $\delayp$\index{$\delayp$}  (Polynomial Time Delay)\index{Polynomial Time Delay},
		\item $\delayps$\index{$\delayps$}  (Polynomial Time Delay with Polynomial Space)\index{Polynomial Time Delay with Polynomial Space}
	\end{enumerate}
	if there is an enumeration algorithm such that
	\begin{enumerate}
		\item  its running time is in $\poly(|x|, |C(x)|)$ for $x$ in $I$,
		\item  its $i$-th delay for $i$ in $[n]$ is in $\poly(|x|, |C^i(x)|)$ for $x$ in $I$,
		\item its $i$-th delay for $i$ in $[n]$ is in $\poly(|x|,|C^i(x)|)$ for $x$ in $I$, and
		\item the same as in 3.~holds and the algorithm requires at most polynomial space in the input size,
	\end{enumerate}
	respectively, where $C^i(x)$ denotes the set of solution that have been output before the $i$-th solution has been output.
\end{definition}

We say that an algorithm is \emph{output-sensitive} if it fulfills condition 1 of the above definition. Moreover, we say that an algorithm has \new{incremental delay} (\new{polynomial delay}) if it fulfills the second (third) condition of the above definition. The following hierarchy result holds.

\begin{theorem}[{\cite[14~sqq.]{Schm2009}}]
\begin{equation*}
\delayps \subseteq \delayp \subseteq \incp \subset \totalp\,.
\end{equation*}
\end{theorem}

An enumeration algorithm will be regarded \new{efficient} if it is output-sensitive.

\subsection{MOCO Problems are Enumeration Problems}\label{mocoenum}

We will now explain, how MOCO problems as defined in Definition~\ref{def:moco} relate to enumeration problems as defined in Definition~\ref{d:ep}.
More formally, we will show that the general MOCO problem is an enumeration problem.
Given a MOCO instance $(\cosolutions, C)$ we need to define the set $I$ of instances of the enumeration problem and the configuration mapping.

To define $I$, we fix a binary representation of $(\cosolutions, C)$ with encoding length of at most $\mathcal{O}(n+\langle C\rangle)$.
Then $I$ is the set of all such encodings of instances of our MOCO problem.
The configuration mapping now maps an instance to its Pareto-front $\YN$, by fixing an arbitrary encoding of rational numbers with at most $\mathcal{O}(\langle r \rangle)$ bits for a rational number $r$.
We observe that the value of each component $j$ of each point of the Pareto-front is at most $\sum_{i=1}^n |C_{ji}|\in \mathcal{O}(\langle C \rangle)$ and thus the encoding of each point in the Pareto-front is polynomial in the instance encoding size.
Thus, the configurations of a MOCO problem are the points of the Pareto-front and not, for example, the solutions of the MOCO problem.

Solving the so defined enumeration problem will give us the set of configurations, i.e., the Pareto-front in the given encoding.
This corresponds exactly to our definition of solving a MOCO problem.

\section{A Sufficient Condition for Output-sensitivity from the MOCO Literature} \label{sec:suff_cond}

In this section, we will show a sufficient condition for a MOCO problem being solvable in output-polynomial time, which follows directly from the multiobjective optimization literature.
A broad subject of investigation in the multiobjective optimization literature is the $\varepsilon$-constraint scalarization:
\begin{align*}
	\min \{ \mathbf{1}^T Cx \mid x \in \cosolutions, C_1^T x \leq \varepsilon_1, \dots, C_d^T x \leq \varepsilon_d\},\, \text{for } \varepsilon\in\QQ^d\,.
\end{align*}
It is well known that the complete Pareto-front can be found using this scalarization, cf.~\citep{CH83}, but the scalarization itself can be hard to solve, cf.~\citep{E06}.
In the past, authors have proven several bounds on the number of $\varepsilon$-constraint scalarizations needed to find the Pareto-front of a general MOCO problem.
It is possible to construct a sufficient condition for enumerating the Pareto-front of a MOCO problem in output-polynomial time from these results.

One of the first such bounds is due to~\citet{LTZ06}, who prove that at most $\mathcal{O}(|\YN|^{d-1})$ single-objective $\varepsilon$-constraint scalarization problems need to be solved.
The currently best known bound is due to~\citet{KLV15}, who prove a bound of $\mathcal{O}(|\YN|^{\lfloor \frac{d}{2} \rfloor})$.

From these results we can formulate the following corollary.
\begin{corollary}\label{thm:suff}
	If the $\varepsilon$-constraint scalarization of a given MOCO problem $P$ can be solved in polynomial time for a fixed number of objectives, $\YN$ can be enumerated in output-polynomial time for a fixed number of objectives.
\end{corollary}
One example for such a problem is the multiobjective global minimum-cut problem.
\citet{AZ04} showed, that the $\varepsilon$-constraint version of the problem can be solved in time $\mathcal{O}(mn^{2d}\log n)$.
Thus, the multiobjective global minimum-cut problem is an example of a multiobjective optimization problem that can be solved in output-polynomial time for each fixed number of objectives.

Now, the question arises if this condition is also necessary, i.e., can we show that a MOCO problem cannot be solved in output-polynomial time by showing that the $\varepsilon$-constraint scalarization cannot be solved in polynomial time.
But as the following proposition shows, the condition of Corollary~\ref{thm:suff} is not necessary, unless $\cp = \cnp$.
\begin{proposition}\label{thm:trad_hard_enum_easy}
	There is a MOCO problem $P\in \incp$ with $\decision{P}$ being $\cnp$-complete.
\end{proposition}
\begin{proof}
	We consider the following biobjective problem
	$$(P)\ \min \{(c^Tx, -c^Tx) \mid x \in \{0,1\}^n \}$$
	with $c\in \NN^n$.
	We will first show $\cnp$-hardness and membership in $\cnp$ of $\decision{P}$ and then show membership in $\incp$ of $P$.
	
	Recall that  $\decision{P}$ is defined as the decision problem in which we are given an instance $(\cosolutions, C)$ of $P$ and a vector $k\in \QQ^d$ and the goal is to decide whether there exists an $x\in\cosolutions$, such that $Cx\leq k$.
	Setting
	\begin{align*}
	k := \frac{1}{2} \mathbf{1}^T c \begin{pmatrix}1\\-1\end{pmatrix},
	\end{align*}
	we can reformulate $\decision{P}$ in the following way: 
	Decide for a given vector $c\in\NN^n$ if there is an $x\in\{0,1\}^n$, such that $c^Tx = \frac{1}{2}\sum_{i=1}^n c_i$.
	We observe that this decision problem is the well-known partition problem shown by~\citet{Gar+1990} to be $\cnp$-hard. It follows from Definition~\ref{def:can_dec_problem} that $\decision{P}$ is in $\cnp$ and thus $\decision{P}$ is $\cnp$-complete.
	
	Enumerating the Pareto-front of $P$ can be done in incremental polynomial time by employing a recursive enumeration scheme on the first $i$ variables (fixing an arbitrary variable ordering).
	Setting all variables to $0$ yields point $(0,0)^T$.
	Now we fix all variables after the $i$-th to $0$ and allow the first $i$ variables to vary.
	Let $F_i$ be the Pareto-front for this restricted problem.
	If we know $F_i$, we can compute $F_{i+1}$ by computing $F_i \cup (F_i+ \{c_{i+1}\})$.
	Each such step yields time $\mathcal{O}(|F_i|\log |F_i|)$ and we have $F_0 \subsetneq F_1 \subsetneq \dots \subsetneq F_n$, i.e., we find at least one new point in each iteration.
\end{proof}

\section{Example of a Hard Problem: Multiobjective Shortest Path} \label{sec:mosp}

In this section, we will classify a problem as not efficiently solvable under mild complexity theoretic assumptions in the sense of output-sensitive complexity, namely the multiobjective shortest path problem, or more precisely, the multiobjective $s$-$t$-path problem.
\begin{definition}[Multiobjective $s$-$t$-path problem (MOSP)]
Given a directed graph $G=(V,E)$ where $E\subseteq V \times V$, two nodes $s,t\in V$, and a multiobjective arc-cost function $c:E \rightarrow \QQ^d_\geq$.
A feasible solution is a path $P$ from node $s$ to node $t$.
We set $c(P) := \sum_{e\in P} c(e)$.
The set of all $s$-$t$-paths is $\mathcal{P}_{s,t}$.
The goal is to enumerate the minima of $c(\mathcal{P}_{s,t})$ with respect to the canonical partial order on vectors.
\end{definition}
It has long been investigated in the multiobjective optimization community.
One of the first systematical studies was by~\citet{H79}.
In practice, MOSP can be solved relatively well.
But usually algorithms which solve MOSP solve a more general problem which we call the \emph{multiobjective single-source shortest-path (MO-SSSP) problem}.
In the MO-SSSP problem, instead of enumerating one Pareto-front of the Pareto-optimal $s$-$t$-paths, we enumerate for every vertex $v\in V \backslash \{s\}$ the Pareto-front of all Pareto-optimal $s$-$v$-paths.
One example of such an algorithm is the well known label setting algorithm by \citet{M84}, which solves the MO-SSSP problem in incremental polynomial time.

We stress here that MOSP and MO-SSSP are indeed different problems.
The insight that the MO-SSSP problem is solvable in incremental polynomial time does not immediately transfer to the MOSP problem.
This can be seen, when using Martins' algorithm for solving MOSP on a given instance, where we cannot even guarantee output-polynomial running time.

More modern label setting and label correcting algorithms usually start from there and try to avoid enumerating too many unneccessary $s$-$v$-paths.
The following considerations show that a large number of these $s$-$v$-paths are needed in these algorithms and that---in contrast to the MO-SSSP problem---there is no ouput-sensitive algorithm solving MOSP in general, if we assume $\cp \neq \cnp$.

Moreover, we will exemplify a general methodology in output-sensitive complexity for showing that there is no output-sensitive algorithm unless $\cp = \cnp$.
Similar to single objective optimization, we show that an enumeration problem is hard by showing the hardness of a decision problem, which is defined as follows.
\begin{definition}[Finished Decision Problem]
Given an enumeration problem $E = (I, C)$ the \new{finished decision problem} $\finished{E}$ is the following problem:
We are given an instance $x\in I$ of the enumeration problem and a subset $M \subseteq C(x)$ and the goal is to decide if $M = C(x)$.
\end{definition}

That is, when we are given an instance $x\in I$ of the enumeration problem $(I,C)$ and a subset $M\subseteq C(x)$ of the configuration set, we ask the question:
Do we already have all configurations?
As commonly done in theoretical computer science, we will identify a decision problem $P$ with its set of ``Yes''-instances.
So for an instance $x$ of $P$, we write $x\in P$ if $x$ is a ``Yes''-instance and $x\notin P$ if $x$ is a ``No''-instance of $P$.

\citet{LLK80} showed that if there is an output-polynomial algorithm for an enumeration problem, then we can solve the associated finished decision problem in polynomial time.
Or equivalently: If the finished decision problem cannot be solved in polynomial time, then we cannot solve the enumeration problem in output-polynomial time.

\begin{lemma}[\cite{LLK80}]\label{lemma:fin}
	If $E \in \totalp$ then $\finished{E} \in \cp$.
\end{lemma}

\begin{proof}
	Let $E=(I, C)$ be given.
	Since $E\in \totalp$, there exists a polynomial function $p$ and a RAM $A$ which enumerates $C(x)$ for a given $x\in I$ in time at most $p(|x|, |C(x)|)$.
	
	We will now construct an algorithm which will decide for a given instance $(x, M)$ of $\finished{E}$ whether it is a ``Yes''- or ``No''-instance.
	We simulate $A$ on $x$ for time $p(|x|, |M|)$.
	If $A$ does not halt on $x$, then $|M| < |C(x)|$ and we can safely answer ``No''.
	If $A$ does halt on $x$, we are almost done.
	It can still be the case that the input $M$ was invalid, i.e., $M \backslash C(x) \neq \emptyset$ or that $M\subsetneq C(x)$ and the algorithm was by chance faster than expected.
	Both conditions can be tested by checking whether $M$ is equal to the output of $A$.
	If it is, we return ``Yes'', otherwise ``No''.
	
	Simulation takes time $\poly(|x|, |M|)$.
	The output of $A$ has at most size $\poly(|x|, |M|)$ and the final check can thus be done in $\poly(|x|, |M|)$.
\end{proof}

We can now use this method to finally show that the MOSP problem is indeed a hard enumeration problem.
For a definition of $\cconp$, we refer the reader to the book by \citet{Aro+2009}.

\begin{theorem}\label{thm:mosp_hard}
	There is no output-sensitive algorithm for the MO $s$-$t$-path problem unless $\cp = \cnp$, even if the input graph is outerplanar.
\end{theorem}

\begin{proof}
We will show that $\finished{\text{MOSP}}$ is $\cconp$-hard.
By Lemma~\ref{lemma:fin} this shows that a $\totalp$ algorithm for MOSP implies $\cp = \cconp$ and thus $\cp = \cnp$.

We reduce instances of the complement of the Knapsack problem:
\begin{align*}
(\text{KP})\ \{ (c^1, c^2, k_1, k_2) \mid  {c^1}^T x \leq k_1, {c^2}^Tx \geq k_2, x \in \{0,1\}^n\}\,.
\end{align*}
Without loss of generality, we can assume that $c^1, c^2 \in \NN^n$, $k_1, k_2 \in \NN$, ${c^1}^T\mathbf{1} > k_1$ and ${c^2}^T\mathbf{1} > k_2$.
The problem above is still $\cnp$-complete under these restrictions, cf.~\citep{KPP04}.

\begin{figure}[tb]
	\begin{center}
	\subfigure{
	\includegraphics[scale=0.5]{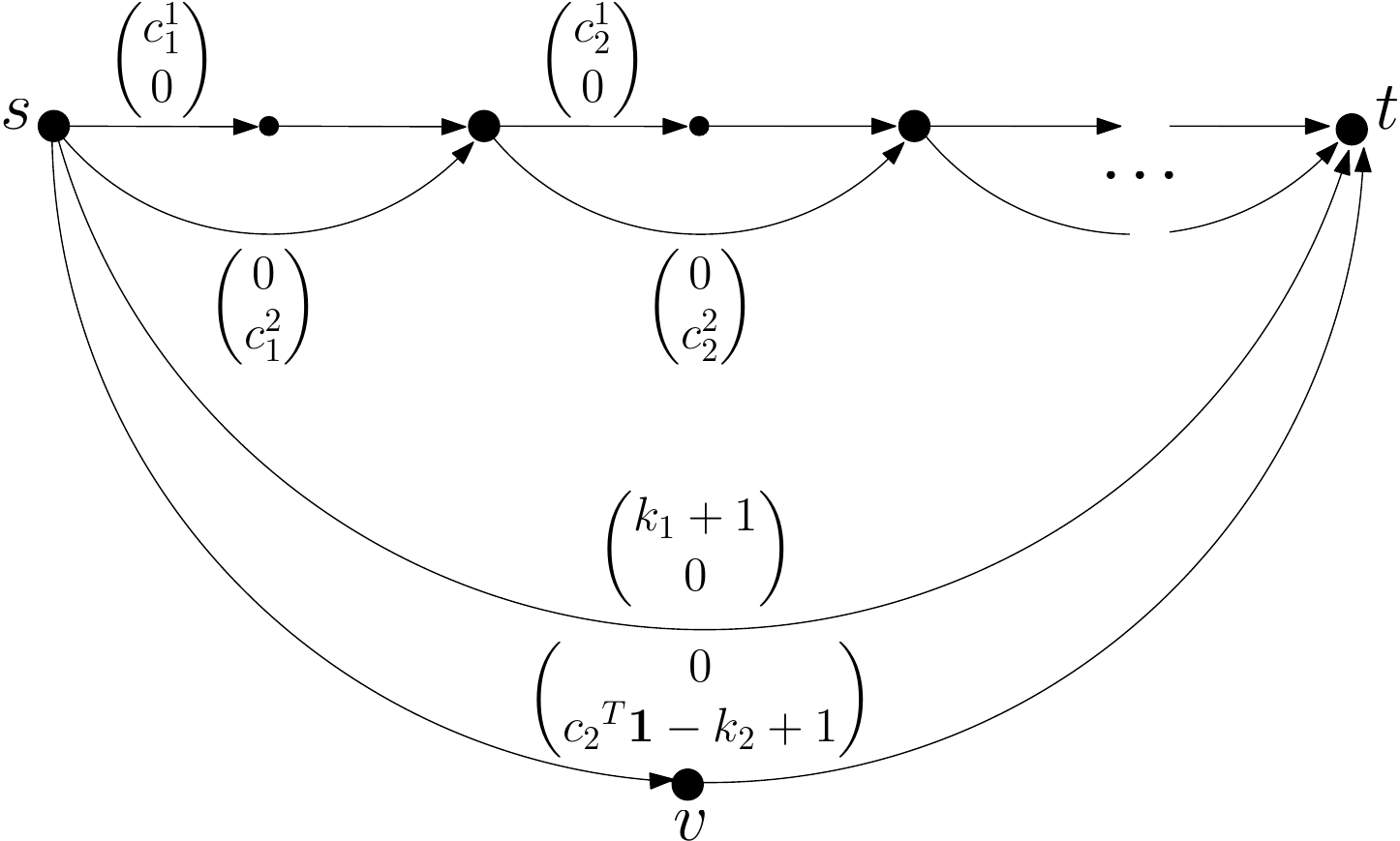}
	}
	\subfigure{
	}
	\end{center}
	\caption{Showing the reduction in the proof for Theorem~\ref{thm:mosp_hard}. Arcs with no label have cost $\mathbf{0}$.}
	\label{fig:mosp}
\end{figure}

We now construct an instance $\hat{I}$ of the $\finished{\text{MOSP}}$ problem from an instance $I$ of the KP problem.
The instance has nodes $\{v_i^1, v_i^2\}$ for every variable $x_i$ and one additional node $v_{n+1}^1$.
It has one arc for every $i\in [1\! :\! n] \colon (v_i^1, v_i^2)$ with cost $(c^1_i , 0)^T$ and for every $i\in [1\! :\! n]$ it has one arc $(v_i^1, v_{i+1}^1)$ with cost $(0, c^2_i)^T$ and one arc $(v_i^2, v_{i+1}^1)$ with cost $\mathbf{0}$.
The node $v_1^1$ is identified with $s$, the node $v_{n+1}^1$ is identified with $t$.
There is one additional arc $(s,t)$ with cost $(k_1 + 1, 0)^T$ and one additional path $(s, v, t)$ with cost $(0, {c^2}^T\mathbf{1} - k_2+1)$.
To complete the reduction, we set $M := \{(k_1 + 1, 0)^T, (0, {c^2}^T\mathbf{1} - k_2+1)^T\}$.
An example of this reduction can be seen in Figure~\ref{fig:mosp}.

We observe that the instance is valid, and that there are at least two Pareto-optimal paths, namely $(s,t)$ and $(s,v,t)$ having cost $(k_1 + 1, 0)^T$ and $(0, {c^2}^T\mathbf{1} - k_2+1)^T$, respectively.
Both paths are Pareto-optimal because $c^1_i,c^2_i > 0 $ for all $i\in [1\! :\! n]$ and thus all other paths have either non-zero components in their objective function values or the sum of all values in one objective function and $0$ in the other.
All steps can be performed in polynomial time in the input instance $I$.

Now, we take an instances $I\in \text{KP}$.
Accordingly, there exists $x\in\{0,1\}^n$ with ${c^1}^Tx \leq k_1$ and ${c^2}^Tx \geq k_2 \Leftrightarrow {c^2}^T(\mathbf{1}-x) \leq {c^2}^T\mathbf{1} - k_2$.
Using this solution, we construct a path $P$ in the MOSP instance $\hat{I}$: For every $i$ with $x_i = 1$, we take the route through node $v^2_i$, inducing cost of $(c^1_i, 0)^T$.
For every $i$ with $x_i = 0$, we take the route directly through arc $(v^1_i, v^1_{i+1})$, inducing cost $(0, c^2_i)^T$.
We observe that $c_1(P) = {c^1}^Tx \leq k_1$ and $c_2(P) = {c^2}^T (\mathbf{1} - x) \leq {c^2}^T\mathbf{1} - k_2$.
But then, $\hat{I}\notin \finished{\text{MOSP}}$, since $P$ is neither dominated by $(s,t)$ nor $(s,v,t)$.

Now, suppose for some instance $I$, the constructed instance $\hat{I} \notin \finished{\text{MOSP}}$, i.e., there is an additional nondominated path $P$ apart from $(s,t)$ and $(s,v,t)$.
Since it is not dominated by $(s,t)$ and $(s,v,t)$, it must hold that $c_1(P) \leq k_1$ and $c_2(P) \leq {c^2}^T\mathbf{1} - k_2$.
But then, we can construct a solution to $\text{KP}$ in $I$ as follows:
The path $P$ cannot take arcs from $(s,t)$ or $(s,v,t)$, so it needs to take the route through the $v_i^1$ and $v_i^2$ nodes.
For every $i\in [1\! :\! n]$ it can only either take arc $(v_i^1, v_i^2)$ or $(v_i^1, v_{i+1}^1)$.
If it takes the first arc, we set $x_i := 1$; if it takes the second arc, we set $x_i := 0$.
This solution then has cost ${c^1}^Tx = c_1(P) \leq k_1$ and ${c^2}^Tx = {c^2}^T\mathbf{1} - c_2(P) \geq {c^2}^T\mathbf{1} - {c^2}^T\mathbf{1} + k_2 = k_2$.
Hence, $I\in \text{KP}$.

This proves that the reduction is a polynomial time reduction from the complement of KP to the finished decision variant of MOSP and thus the theorem.
\end{proof}

The proof also shows that deciding whether the Pareto-front of a MO $s$-$t$-path instance is larger than $2$ is $\cnp$-hard.
Moreover, it gives us a lower bound on the quality we can approximate the size of the Pareto-front.
\begin{corollary}
	It is $\cconp$-hard to approximate the size of the Pareto-front of the MO $s$-$t$-path problem within a factor better than $\frac{3}{2}$.
\end{corollary}

Since MOSP can be seen as a special case of many important MOCO problems, e.g., the multiobjective minimum perfect matching problem and the multiobjective minimum-cost flow problem, these problems also turn out to be hard.

\section{Enumerating Supported Efficient Solutions and Nondominated Extreme Points} \label{sec:supported}

In the history of multiobjective optimization, it has been repeatedly observed that in the case of two objectives a certain subset of the Pareto-front can be enumerated efficiently.
This subset is usually defined as the set of Pareto-optimal solutions that can be obtained by using \new{weighted-sum scalarizations} of a MOCO problem $P$, i.e.,
\begin{align*}
	(\ws{P}{\ell})\ \min \{\ell^TCx \mid x \in \cosolutions\}, \text{ for }\ell \geq 0\,.
\end{align*}

One of the pioneering works in this respect was the discovery of the \emph{Dichotomic Approach} (DA) independently by~\citet{Ane+1979, Coh1978, D79}.
The algorithm finds the \emph{extreme supported} points of the Pareto-front of a biobjective combinatorial optimization problem, i.e., the nondominated extreme points of $\conv c(\cosolutions)$.
We will give different definitions for these terms later in this section.
Especially, Aneja and Nair showed that if we have access to an algorithm $A$ which solves a lexicographic version of the MOCO problem, we can find the set of nondominated extreme points of size $k\geq 2$ by needing $2k -1$ calls to $A$.
The theory of output-sensitive complexity is able to add more ground to the intuition that the set of nondominated extreme points can be computed efficiently for every fixed $d\geq 2$.

A more general view on supported Pareto-optimal solutions can be obtained by extending the view to \new{multiobjective linear programming}.
\begin{definition}[Multiobjective Linear Programming (MOLP)]\label{def:molp}
	Given matrices $A\in \QQ^{m\times n}, C\in\QQ^{d\times n}$ and a vector $b\in\QQ^m$, enumerate the extreme points $\YX$ of the polyhedron $\mathcal{P} := \{Cx \mid Ax\geq b\} + \RR^d_\geq$. We also write an MOLP in the form
	\begin{align*}
		\min\, &Cx\\
		&Ax \geq b\,.
	\end{align*}
\end{definition}
We call the polyhedron $\mathcal{P}$ the \new{upper image} of a given MOLP instance, cf.~also \citep{HLR13}), and set $\Y := \{Cx \mid Ax \geq b\}$.
In the definition of the MOLP problem, we again only allow rational numbers, as real numbers cannot be encoded in our model of computation.

We stress here that this is not the only possible definition of MOLP.
One could also be interested in finding the nondominated facets of the upper image and this would result in a different computational problem.
But for the study of extreme points of MOCO problems, the above problem will serve us better.

We observe that there is a difference between the above version of the MOLP problem and the problem of enumerating all Pareto-optimal basic feasible solutions:
The latter problem cannot be solved in output-polynomial time unless $\cp = \cnp$, because it subsumes the vertex enumeration problem even when allowing only two objectives. \citet{KBBEG08} proved that the vertex enumeration problem for general polyhedra cannot be solved in output-polynomial time unless $\cp = \cnp$.
But we conjecture that there exists an output-sensitive algorithm for the MOLP problem as in Definition~\ref{def:molp}.

The central connection between supported efficient solutions and supported nondominated points and MOLP lies in the concept of convex relaxations of MOCO problems, cf. \citep{EG07,CPG15}.
\begin{definition}[Convex Relaxation]\label{def:convex_relaxation}
	Given a MOCO problem $P=(\cosolutions, C)$, the \new{convex relaxation} of $P$ is the MOLP
\begin{align*}
\min\,&Cx \\
&x \in \conv \cosolutions\,.
\end{align*}
\end{definition}
Definition~\ref{def:convex_relaxation} hides the fact, that a convex relaxation is not in the correct form of Definition~\ref{def:molp}.
Although the convex relaxation of a MOCO problem exists in general, it might be computationally hard to obtain a representation of the MOLP as in Definition~\ref{def:molp} (i.e., an inequality representation), as it might have a large number of constraints or it might have large numbers in the matrix $A$.
But the convex relaxation will serve as a theoretical vehicle to define the extreme points of a MOCO problem and show some connections between MOCO problems and MOLP.

Accordingly, we define the \emph{nondominated extreme points} $\YX$ of a MOCO problem $P$ to be the extreme points of the upper image  of the associated convex relaxation.
An efficient solution $x$ to a MOCO problem $P$ is called \emph{supported}, if $x$ is an efficient solution to the convex relaxation of $P$.


In the following sections, we will first review some literature to shed some lights on solving MOLP in the output-sensitivity framework in Section~\ref{sec:molp}.
Then in Section~\ref{sec:extreme} we are concerned with finding nondominated extreme points to MOCO problems and show how the results from Section~\ref{sec:molp} can be applied to MOCO problems using the relation between a MOCO problem and its convex relaxation.

\subsection{Multiobjective Linear Programming} \label{sec:molp}

There are many ways of solving MOLP problems in the multiobjective optimization literature.
But performance guarantees are seldomly given.
Recently, Bökler and Mutzel proved that MOLP, under some restrictions, is solvable efficiently in the theory of output-sensitive complexity~\citep{BM15}.
We will first review this result and later show that in some special cases better running times can be achieved.

\subsubsection{The General Case with an Arbitrary Number of Objectives}\label{sec:benson}

In \citep{BM15}, the authors prove the following theorem: If the ideal point of an MOLP exists, we can enumerate the extreme points of the upper image in output-polynomial time. 

\begin{theorem}[Enumeration of Extreme Points of MOLP having an ideal point~\citep{BM15}]\label{thm:molp}
	If an ideal point of an MOLP $(C, A, b)$  exists, the extreme points $\YX$ of $\mathcal{P}$ can be enumerated in time $\mathcal{O}(|\YX|^{\lfloor \frac{d}{2} \rfloor}(\poly(\langle A\rangle, \langle C\rangle) + |\YX | \log |\YX |))$ for each fixed number $d$ of objectives.
\end{theorem}

The ideal point of an MOLP is the point $(\min \{ C_i x \mid x \in \cosolutions\})_{i\in [d]}$, where $C_i$ is the $i$-th row of $C$.
We remark that an ideal point does not exist for every MOLP, even though the MOLP is feasible, e.g., if one of the linear programming problems $\min \{ C_i x \mid x \in \cosolutions\}$ does not have a finite optimal value.

To achieve this result, the authors conducted a running time analysis of a variant of Benson's algorithm.
Benson's algorithm was proposed by~\citet{B98} and the mentioned variant, called the dual Benson algorithm, was proposed by~\citet{Ehr+2012}.

From a highlevel point of view, the algorithm computes a representation of a polyhedron $\mathcal{D}$, called \new{lower image}, which is geometrically dual to the upper image $\mathcal{P}$.
The theory of geometric duality of MOLP was introduced by~\citet{Hey+2008}, and in some sense discovered independently by \citet{Prz+2010}.
The representation consists of the extreme points and facet inequalities which define $\mathcal{D}$.
Because of geometric duality of MOLP, the extreme points and facet inequalities of $\mathcal{D}$ correspond to the facet inequalities and extreme points of $\mathcal{P}$, respectively, and can be efficiently computed from them.

To compute the extreme points and facet inequalities of $\mathcal{D}$, the algorithm first computes a polyhedron $P_0$ which contains $\mathcal{D}$ entirely.
Then, in each iteration $i$, it discovers one extreme point $v$ of the polyhedron $P_i$ and checks if $v$ lies on the boundary of $\mathcal{D}$.
If this is the case, the point is saved as a candidate extreme point of $\mathcal{D}$.
If this is not the case, the algorithm finds an inequality which separates $v$ from $\mathcal{D}$ and supports $\mathcal{D}$ in a face.
Then, $P_{i+1}$ is computed by intersecting $P_i$ with the halfspace defined by the inequality found.
In the end, redundant points are removed from the set of candidates and the remaining set is output.

To check if $v$ lies on the boundary of $\mathcal{D}$ and to compute the separating hyperplane, we only have to solve a weighted-sum LP of the original MOLP P:
\begin{align*}
	(\ws{P}{\ell})\ \min\ \{\ell^TCx \mid Ax \geq b \},
\end{align*}
where $\ell\in \QQ^d$ and $||\ell||_1 = 1$.
Another important subproblem to solve is a variant of the well-known vertex enumeration problem.

To arrive at the running time of Theorem~\ref{thm:molp}, Bökler and Mutzel also provide an improvement to the algorithm by guaranteeing to find facet supporting inequalities exclusively.
This is possible at the expense of solving a lexicographic linear programming problem
\begin{align*}
	(\lws{P}{\ell})\ \lexmin\ \{(\ell^TCx, c_1x, \dots, c_dx) \mid Ax \geq b \},
\end{align*}
which turns out to be polynomial-time solvable in general.

\subsubsection{The Special Case with Two Objectives}

Now let $P$ be a biobjective linear program. In the following we present a polynomial delay algorithm, which only needs polynomial space in the size of the input, for enumerating nondominated extreme points. We assume that we have access to a solver for the lexicographic linear programming problem. The idea of the algorithm is the following: We start by computing a nondominated extreme point $y^1$ that minimizes the second objective, i.e., we solve $\lws{P}{(0,1)}$. Then, by solving a lexicographic linear program, a facet $F$ of $\cY$ is computed that supports $y^1$ such that $y^1$ is the lexicographic maximum of $F$. By solving another lexicographic linear program, the algorithm finds the other nondominated  extreme point $y^*$ that is also supported by $F$, see Figure~\ref{fig:delayps} for a geometric illustration. We then restart the procedure with $y^*$. The algorithm terminates when the nondominated extreme point is found that minimizes the first objective.

\begin{figure}[H]
\begin{center}
\includegraphics[scale=0.7]{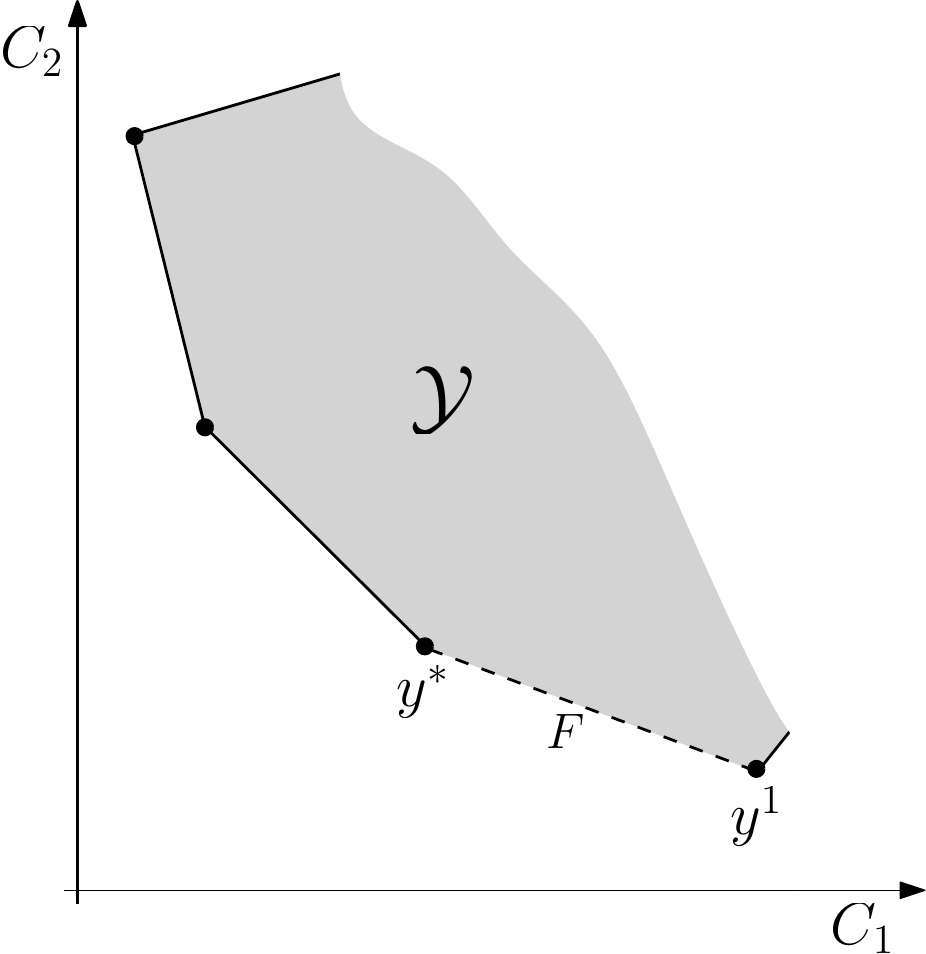}
\end{center}
\caption{Illustration of the idea of the algorithm with polynomial delay and polynomial space.}
\label{fig:delayps}
\end{figure}

In order to compute a facet that supports a nondominated extreme point $y$ of $P$, we need the LP $D_2(y)$, cf.~\citep{Ise1974}:
\begin{equation*}\label{e:d2}\tag{$D_2(y)$}
\begin{aligned}
&\text{maximize} & & b\trans u - y\trans \lambda\\
&\text{subject to} & & (u, \lambda) \geq  \mathbf{0}^{m+p} \\
&&& A\trans u = C\trans \lambda \\
&&& (\mathbf{1}^{p})\trans \lambda = 1\\
&&& (u, \lambda) \in \bbQ^{m + p}\,.
\end{aligned}
\end{equation*}

We get the following result.

\begin{lemma}\label{l:wf}
Let $P$ be a biobjective linear program, let $y$ be a nondominated extreme point for $P$ and let $(u, \lambda)$ be a maximum for $D_2(y)$ such that there is no other maximum $(\bar{u}, \bar{\lambda})$ for $D_2(y)$ so that  $\lambda <_{\lex} \bar{\lambda}$. Then $F := \lb y \in \bbR^2 \mathrel{\big|} \lambda\trans y = b\trans u \rb$ is a facet of $\cY$ that supports $\cY$ in $y$ such that $y$ is the lexicographic maximum of $F$.
\end{lemma}
\begin{proof}
The result easily follows from Proposition 4.2 in~\citet{Ehr+2012}.
\end{proof}

Moreover, in order to compute the other extreme point of $F$, we can simply add the constraint
\begin{equation*}
\lambda\trans Cx = b\trans u
\end{equation*}
to $P$ and solve the lexicographic linear programm for $\lws{P}{\lambda}$, where $(u, \lambda)$ is a maximum of $D_2(y)$ as stated in Lemma~\ref{l:wf}. We get the following result.

\begin{theorem}\label{pspd}
Let $P$ be a biobjective linear program, then the above described algorithm outputs the nondominated extreme points of $P$ with polynomial delay \emph{and} polynomial space.
\end{theorem}
\begin{proof}
Let $y^*$ in $\YX$ and let $(u, \lambda)$ be a lexicographic maximum for the lexicographic linear program $D_2(y)$ as stated in Lemma~\ref{l:wf}. Then, by Lemma~\ref{l:wf}, $F := \lb y \in \bbR^2 \mid \lambda\trans y = b\trans u \rb$ is a facet of $\cY$ with $y$ in $F$ such that $y$ is a lexicographic maximum of $F$. At some point, the algorithm will find the point that minimizes the first objective as it is the lexicographic minimum of a facet of $\cY$. Hence, the algorithm is finite.
 
By~\citet{BM15}, each iteration can be computed in polynomial time. In each iteration one nondominated extreme point is output, therefore the $i$-th delay of the algorithm is polynomial in the input size. Moreover, observe that at most three nondominated extreme points have to be kept in memory during each delay. Hence, the space is bounded polynomially in the input size.
\end{proof}

\subsection{Enumerating Supported Nondominated Points}\label{sec:extreme}

Now, we will come back to MOCO problems and show how we can find supported nondominated and extreme nondominated points.
In Section~\ref{sec:bm15} we show how we can make use of the results from Section~\ref{sec:benson}, even when we do not have a compact LP formulation of our MOCO problem.
We show that in the case of two objectives, we can again achieve better running times in Section~\ref{sec:bmm15}.
In the end, we discuss a result by~\citet{OU07}, where the authors are concerned with finding supported spanning trees.

\subsubsection{Enumerating Extreme Points of MOCO with an Arbitrary Number of Objectives}\label{sec:bm15}

To find nondominated extreme points we can apply the algorithms from Section~\ref{sec:benson} to the convex relaxation of the MOCO we want to solve.
On one hand, convex relaxations of MOCO problems have an ideal point if and only if the MOCO instance is feasible.

On the other hand, because of the aforementioned problems, we cannot easily compute the convex relaxation of a MOCO problem to apply the Dual Benson algorithm.
But the only interaction between the algorithm and the MOLP at hand lies in solving weighted-sum linear programs.
Hence, instead of solving the weighted-sum objective over the feasible set of the convex relaxation, we can optimize the weighted-sum objective over the feasible set of the MOCO problem.

These considerations lead us to the following theorem.

\begin{theorem}[Enumeration of Extreme Points of MOCO Problems I~\citep{BM15}]\label{thm:poly_total}
	For every MOCO problem $P$ with a fixed number of objectives, the set of extreme nondominated points of $P$ can be enumerated in output-polynomial time if we can solve the weighted-sum scalarization of $P$ in polynomial time.
\end{theorem}

Moreover, if we are able to solve lexicographic objective functions over the feasible set of the MOCO problem, we can even bound the delay of the algorithm.

\begin{theorem}[Enumeration of Extreme Points of MOCO Problems II~\citep{BM15}]\label{inc_poly}
	For every MOCO problem $P$ with a fixed number of objectives, the set of extreme nondominated points of $P$ can be enumerated in incremental polynomial time if we can solve the lexicographic version of $P$ in polynomial time.
\end{theorem}

One example of a problem where it is hard to compute the whole Pareto-front, but enumerating the nondominated extreme points can be done in incremental polynomial time is thus the MOSP problem although both the size of the Pareto-front and the number of nondominated extreme points can be superpolynomial in the input size, cf. \citep{G80, C83}.

There are two further methods which are concerned with finding extreme nondominated points of the Pareto-front of general multiobjective integer problems: The methods by \cite{Prz+2010} and \cite{OK10}.
Both do not give any running time guarantees, but the investigation of these can be subject to further research.

\subsubsection{Enumerating Extreme Points of Bicriterial Combinatorial Optimization Problems}\label{sec:bmm15}

In the following, we investigate different variants of the DA, i.e., with and without access to an algorithm to solve the lexicographic version of a biobjective combinatorial optimization (BOCO) problem $P$. Especially, we show that the nondominated extreme points of $P$ can be enumerated with polynomial delay if we have access to an algorithm to solve the lexicographic version of $P$. 

\subsubsection{The Dichotomic Approach}

Let $P$ be a BOCO problem. We begin by describing the DA without access to an algorithm to solve the lexicographic version of $P$. Notice that~\citet{Ane+1979} assume access to such an algorithm. 

In a first step, the DA computes optimal solutions $y^0$ and $y^1$ to the weighted-sum scalarizations $\ws{P}{(1,0)}$ and $\ws{P}{(0,1)}$, respectively. If $y^0$ equals $y^1$, the algorithm terminates since the ideal point is a member of the Pareto-front. Otherwise, the pair $(y^0,y^1)$ is added to the queue $\cQ$.

In each iteration, the algorithm retrieves a pair $(y^2,y^3)$ from $\cQ$ and $\lambda$ is computed such that
\begin{equation}\label{eq:lambda}
\lambda\trans y^2 = \lambda\trans y^3\,.
\end{equation}

Hence, the algorithm computes $\lambda = (|y^2_2 - y^3_2|, |y^2_1 - y^3_1|)$. The algorithm proceeds by computing an optimal solution $\bar{y}$ to the weighted-sum scalarization $\ws{P}{\lambda}$. If $\bar{y}$ is located in the line segment spanned by $y^2$ and $y^3$, i.e., $\lambda\trans y^2 =  \lambda\trans \bar{y} = \lambda\trans y^3$, the algorithm neglects $\bar{y}$ as it is either not a nondominated extreme point or was already discovered by the algorithm in a previous iteration, see Figure~\ref{fig:dac1}. Otherwise, $\bar{y}$ is added to the set $\cA$ and the pairs $(\bar{y},y^2)$ and $(\bar{y},y^3)$ are added to $\cQ$, see Figure~\ref{fig:dac2}. The algorithm terminates when $\cQ$ is empty. Moreover, the nondominated non-extreme points in $\cA$ are removed and the remaining points are output. We get the following result.

\begin{figure}
\begin{center}
\subfigure[Case 1]{
\includegraphics[scale=0.7]{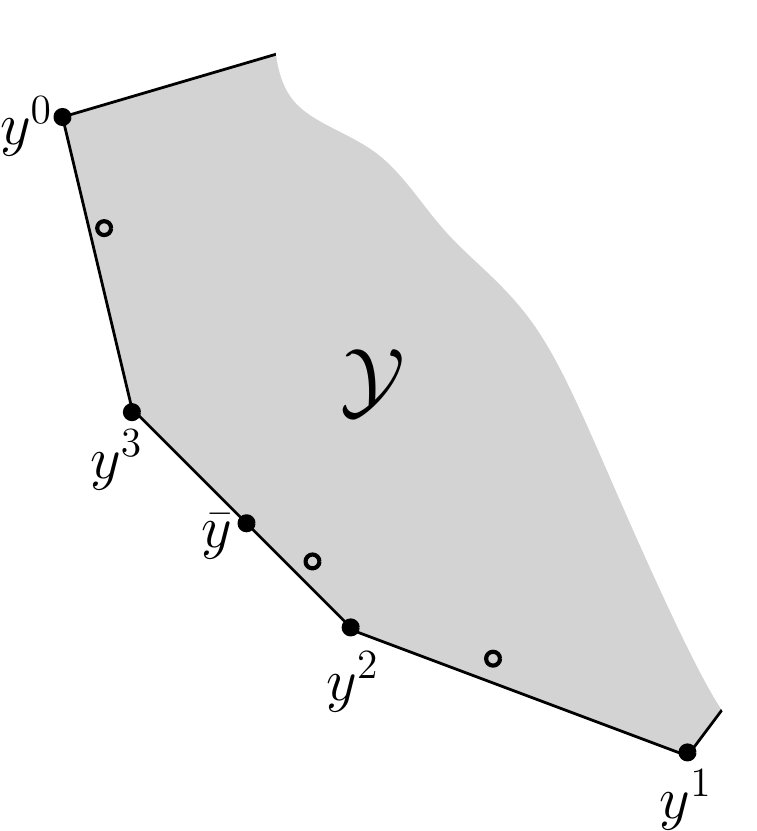}
\label{fig:dac1}
}
\subfigure[Case 2]{
\includegraphics[scale=0.7]{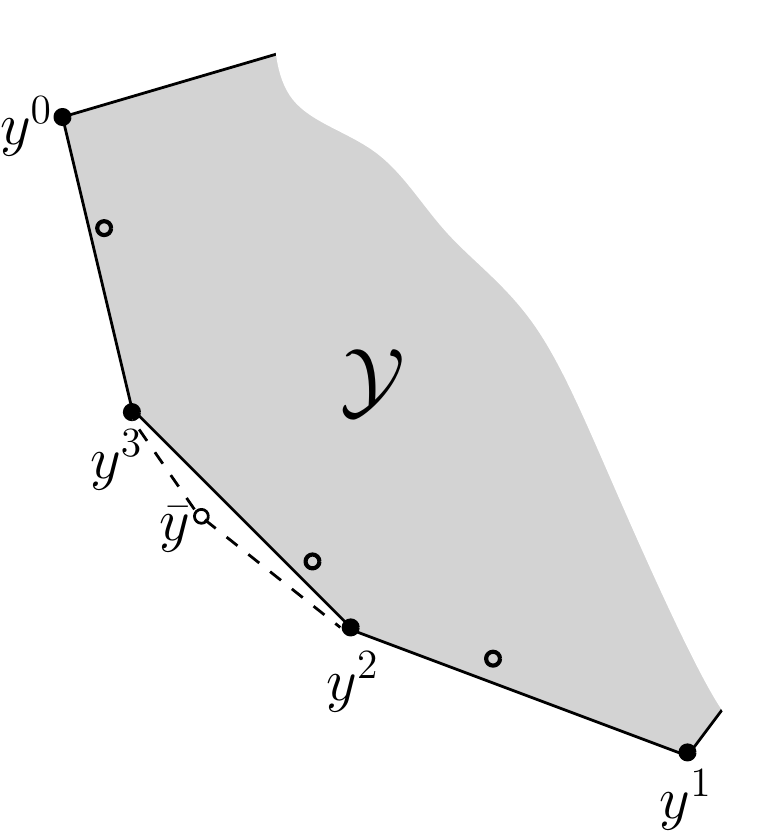}
\label{fig:dac2}
}
\end{center}
\caption{Illustration of the idea of the DA}
\end{figure}

\begin{theorem}\label{t:da}
Let $P$ be a BOCO problem, then the DA without access to an algorithm to solve the lexicographic version of $P$ outputs the nondominated extreme points of $P$. Moreover, if the weighted-sum scalarization of $P$ is solvable in polynomial time, the algorithm is output-sensitive.
\end{theorem}
\begin{proof}
First, observe that there is no nondominated extreme point $y$ such that $y_1 < y_1^0$ and $y_2 < y^1_2$, where $y^0$ and $y^1$ are the points found in the first step of the algorithm. Secondly, observe that there is at most one nondominated non-extreme point found by the algorithm per facet of $\cY$. Hence, the algorithm is finite. Moreover, by~\citet{Ise1974} all points computed by the algorithm except possibly $y^0$ and $y^1$ are nondominated.

Now assume there is a nondominated extreme point $\bar{y}$ that is not found by the algorithm. By the first observation, $y^0 <_{\lex} \bar{y} <_{\lex} y^1$. Hence, there must be a pair of points $(y^2,y^3)$ such that $y^2 <_{\lex} \bar{y} <_{\lex} y^3$ and there is no other pair $(y^4, y^5)$ that is ``closer'' to $\bar{y}$, i.e., $y^2 <_{\lex}y^4 <_{\lex} \bar{y} <_{\lex} y^5 <_{\lex} y^3$. Since the algorithm investigates each such pair, the algorithm will eventually discover $\bar{y}$ and we arrive at a contradiction. Hence, the algorithm finds all nondominated extreme points of $P$.

Because of the second observation and the fact that the number of nondominated facets of $\cY$ is linearly bounded in the cardinality of $\YX$, the number of weighted-sum scalarizations the algorithm has to solve is linearly bounded in the cardinality of $\YX$. Moreover, the encoding lengths of the points discovered during the computation of the algorithm are polynomially bounded in the size of the input, cf.~\citep{BM15}. Hence, the algorithm is output-sensitive if the weighted-sum scalarization of $P$ is solvable in polynomial time. 
\end{proof}

The above algorithm does not lead directly to an algorithm with polynomial delay since there may exist several solution to the weighted-sum scalarization $\ws{P}{\lambda}$, which are not necessarily extreme. However, if we assume access to an algorithm to solve the lexicographic version of $P$, we get the following result. 

\begin{corollary}
Let $P$ be a BOCO problem, then the DA with access to an algorithm to solve the lexicographic version of $P$ enumerates the nondominated extreme points of $P$. Moreover, if the lexicographic version of $P$ is solvable in polynomial time, the points are enumerated with incremental delay.
\end{corollary}
\begin{proof}
The correctness follows from the proof of Theorem~\ref{t:da}. Since we have access to an algorithm to solve the lexicographic version of $P$, each point found by the algorithm is nondominated extreme. Hence, the running time to discover a new nondominated extreme point or to conclude that no such point exists is linearly bounded in the number of the nondominated extreme points found so far. Therefore, the algorithm enumerates the nondominated extreme points with incremental delay if the lexicographic version of $P$ is solvable in polynomial time.
\end{proof}

Note that a lexicographic version of many BOCO problems can be solved as fast as the single objective problem.
For example the biobjective shortest path, matroid optimization and assignment problems.
Also if a compact formulation is available, the lexicographic weighted sum problem can be solved as a lexicographic linear programming problem, which can also be solved as fast as a single objective linear programming problem (cf. \cite{BM15}).

\paragraph{A Variant of the Dichotomic Approach with Polynomial Delay}

We present a simple modification of the DA with access to an algorithm to solve the lexicographic version of $P$. The idea is the following: When a new nondominated extreme point is found, the DA puts two pairs of nondominated extreme points into the queue $\cQ$. At some point, the orginal algorithm checks if there is any nondominated extreme point $\bar{y}$ ``between'' such a pair $(y^2,y^3)$ of points, i.e., $y^2 <_{\lex} \bar{y} <_{\lex} y^3$ or $y^3 <_{\lex} \bar{y} <_{\lex} y^2$. Now instead the modified algorithm computes $\lambda$ as in (\ref{eq:lambda}) and solves $\lws{P}{\lambda}$ resulting in a nondominated extreme point $y$. If $y$ equals $y^2$ or $y^3$, we have verified that there is no point between the two points. Otherwise we add the triple $(y,y^2,y^3)$ to the queue $\widetilde{\cQ}$. 
If such a triple is retrieved from $\widetilde{\cQ}$, we can output the nondominated extreme point $y$. The rest of the algorithm is analogous to original Dichotomic Approach. Therefore, in each iteration exactly one nondominated extreme point is output. Moreover, in each iteration at most two lexicographic weighted sum-scalarization problems have to be solved. Hence, we can easily derive the following result.

\begin{corollary}
Let $P$ be a BOCO problem, then the above described variant of the DA outputs the nondominated extreme points of $P$ with polynomial delay if the lexicographic version of $P$ is solvable in polynomial time.
\end{corollary}

We remark here that we obtain a polynomial delay algorithm by keeping back points which we would have presented to the user already in the original DA.
This is not what we usually want to do in practice.
To overcome this issue, one could investigate the amortized delay, i.e., the amortized time needed per output. Moreover, we get the following result which follows from Theorem~\ref{pspd}.

\begin{corollary}
Let $P$ be a BOCO problem. If there is a \new{compact formulation} of $P$, then the above described algorithm outputs the nondominated extreme points of $P$ with polynomial delay and polynomial space.
\end{corollary}

\subsubsection{Enumerating Supported Efficient Spanning Trees} \label{sec:ou07}

The first paper discussing techniques from enumeration algorithmics in the context of multiobjective optimization is by~\citet{OU07}.
They propose a solution to the multiobjective spanning tree problem, which is defined as follows.
\begin{definition}[Multiobjective Spanning Tree Problem (MOST)]
	Given an undirected Graph $G=(V,E)$, where $E\subseteq \{ e \subset V \mid |e| = 2\}$, and a multiobjective edge-cost function $c:E \rightarrow \QQ^d_\geq$.
	A feasible solution is a \new{spanning tree} $T$ of $G$, i.e., an acyclic, connected subgraph of $G$ containing all nodes $V$.
	We set $c(T) := \sum_{e\in T} c(e)$.
	The set of all such trees is denoted by $\mathcal{T}$.
	The goal is to enumerate the minima of $c(\mathcal{T})$ with respect to the canonical partial order on vectors.
\end{definition}
However, the model considered by~\citet{OU07} is different from the one considered in this paper.
They consider enumerating all \new{supported weakly efficient} spanning trees, i.e., all $T\in \mathcal{T}$, such that $T$ is an optimal solution to $\ws{\text{MOST}}{\ell}$ for some $\ell \geq 0$.
This includes non-Pareto-optimal spanning trees for $\ell\in \{e_i \mid i \in [d]\}$ and also equivalent trees $T$ and $T'$ with $T \neq T'$ and $c(T) = c(T')$.

For this problem, the authors propose an algorithm which runs with polynomial delay by employing a variant of the reverse search method by~\citet{AF92}.
The authors also generalize this algorithm to the case of enumerating Pareto-optimal basic feasible solutions of an MOLP if it is nondegenerate, i.e., every basic solution has exactly $n$ active inequalities.

\section{Conclusion}

In this paper, we showed that output-sensitive complexity is a useful tool to investigate the complexity of multiobjective combinatorial optimization problems.
We showed that in contrast to traditional computational complexity in multiobjective optimization, output-sensitive complexity is able to separate efficiently solvable from presumably not efficiently solvable problems.

On one hand, we provided examples of problems for which output-sensitive algorithms exist, e.g., the multiobjective global minimum-cut problem, and also computing the nondominated extreme points for many MOCO problems can be done efficiently.

On the other hand, we demonstrated that it is also possible to show that a MOCO problem does not admit an output-sensitive algorithm under weak complexity theoretic assumptions as $\cp \neq \cnp$.
One example is the multiobjective shortest path problem, which in turn also rules out the existence of output-sensitive algorithms for the multiobjective versions of the minimum perfect matching and the minimum-cost flow problem.

To conclude, we will raise some questions which can be investigated in further work.

\paragraph{Multiobjective Linear Programming}
While in Section~\ref{sec:bm15} it was shown that we can enumerate the extreme points of the upper image of an MOLP in incremental polynomial time, we needed the assumption that an ideal point exists.
We are looking forward to being able to lift this assumption to the general case of MOLP and obtain at least an output-sensitive algorithm.

Another question is, if it is necessary to fix the number of objectives $d$ in the running time to obtain an output-sensitive algorithm.
We conjecture that this is actually the case when using standard assumptions from complexity theory and will investigate this in the future.

A longer reaching question is, if we can find an algorithm with polynomial delay for the case of general MOLP and a fixed number of objectives.

\paragraph{Multiobjective Combinatorial Optimization}
Not much is known in this regard.
A very intriguing problem is the following:
\begin{definition}[Unconstrained Biobjective Combinatorial Optimization Problem]
	Given two vectors $c^1, c^2 \in \QQ^n$.
	The set of feasible solutions is $\{0,1\}^n$ and the objective function is
	\begin{align*}
	c: x \mapsto \begin{pmatrix} {c^1}^Tx \\ {c^2}^Tx \end{pmatrix}.
	\end{align*}
	The goal is to enumerate the minima of $c(\{0,1\}^n)$ with respect to the canonical partial order on vectors.
\end{definition}
Even for this problem, which is an unconstrained version of the biobjective knapsack problem, it is unkown if an output-sensitive algorithm exists.
Moreover, if there is no output-sensitive algorithm for this problem we can rule out output-sensitive algorithms for many MOCO problems, e.g., the multiobjective assignment and the multiobjective spanning tree problem.

\paragraph{Additional Complexity Classes}
There has been some research on new complexity classes, combining fixed parameter tractability theory and output-sensitive complexity theory, which is \emph{Parameterized Enumeration} by \citet{CMMSV13}.
It is very interesting to see if parameters can be found which make MOCO problems hard, apart from the fact that multiple objectives are present.

\printbibliography
\end{document}